\documentclass[12pt]{amsart}

\usepackage{amsmath}
\usepackage{hyperref}
\usepackage[utf8]{inputenc} 
\usepackage[ruled,noline,noend]{algorithm2e}
\usepackage{mathtools}
\usepackage{enumerate}
\usepackage[nice]{nicefrac}

\newtheorem{theorem}{Theorem}[section]

\newtheorem{lemma}[theorem]{Lemma}
\newtheorem*{claim}{Claim}

\newtheorem{conjecture}[theorem]{Conjecture}

\DeclarePairedDelimiter{\ceil}{\lceil}{\rceil}
\DeclarePairedDelimiter{\floor}{\lfloor}{\rfloor}

\begin{document}
	\sloppy
	
	\title[Rota's Basis Conjecture]{Improved bounds for Rota's Basis Conjecture}
	
	\author{Sally Dong}
	\author{Jim Geelen}
	\address{Department of Combinatorics and Optimization, University of Waterloo, Waterloo, Canada} 
	\thanks{This research was partially supported by grants from the
Office of Naval Research [N00014-10-1-0851] and NSERC [203110-2011] as well as an Undergraduate Student Research Award 
from NSERC}
	
	\date{\today}
	
	\begin{abstract}
		We prove that, if $B_1, \dots, B_n$ are disjoint bases of a rank-$n$ matroid, then there are at least $\frac{n}{7 \log n}$ disjoint transversals of $(B_1, \dots, B_n)$ that are also bases.
	\end{abstract}
	\maketitle
	
	\section{Introduction}
	
	A  \emph{transversal basis} of a collection $(B_1, \dots, B_n)$ of sets of elements in a rank-$n$ matroid is a basis  containing exactly one element from each of $B_1,\ldots,B_n$. Rota's Basis Conjecture, which first appeared in~[\ref{cite:HR94}], is receiving renewed interest as the topic of Polymath 12 ~[\ref{cite:polymath12}]. 
	
	\begin{conjecture}[Rota's Basis Conjecture]
		Given disjoint bases $B_1, \dots, B_n$ of a rank-$n$ matroid, there exist $n$ disjoint transversal bases. 
	\end{conjecture}
	
	Geelen and Webb~[\ref{cite:GW07}] showed that it is possible to get $\ceil{\sqrt{n-1}}$ disjoint transversal bases; our main result improves on their bound.
	
	\begin{theorem} \label{thm:main}
		Given disjoint bases $B_1, \dots, B_n$ of a rank-$n$ matroid, where $n\ge 2$, there are at least $\floor*{\frac {n}{6 \ceil{\log n}}}$ disjoint transversal bases.
	\end{theorem}
	
	Throughout the paper we use the natural logarithm. Using the same methods, but taking more care with the calculations, it should be possible to improve on our bound of $\floor*{\frac{n}{6 \ceil{\log n}}}$; however, 
	new ideas will be needed to beat $\frac{n}{\log n}$. The bound of $\frac{n}{7 \log n}$,  claimed in the abstract, is obtained by combining the bound
        $\ceil{\sqrt{n-1}}$, when $n\le 3000$, with the bound $\floor*{\frac {n}{6 \ceil{\log n}}}$, when $n> 3000$.
		
	We present the central ideas of the proof here in the introduction, leaving the technical details for the next section. 
	We deduce Theorem~\ref{thm:main} from the following key result.
	
	\begin{theorem} \label{thm:main2}
		Let $B_1, \dots, B_n$ be disjoint bases of a rank-$n$ matroid, where $n\ge 2$, and let $\alpha = 3 \ceil{\log n}$. If we choose $\alpha$-element subsets $S_1, \dots, S_n$ independently and uniformly at random from $B_1,\ldots,B_n$, respectively, then $(S_1, \dots, S_n)$ contains a transversal basis with probability at least $\nicefrac 12$.
	\end{theorem}
	
	We start by showing that Theorem~\ref{thm:main2} implies Theorem~\ref{thm:main}.
	
	\begin{proof}[Proof of Theorem \ref{thm:main}] Let $m = \floor*{\frac{n}{6 \ceil{\log n}}}$. For each $i \in \{1, \dots, n\}$, 
	let $S_{i,1}, \dots, S_{i,2m}$ be disjoint $\alpha$-element subsets of $B_i$ chosen at random.
	For each $j \in \{1, \dots, 2m\}$, the sets $S_{1,j}, \dots, S_{n,j}$ are subsets of $B_1,\ldots,B_n$ that are chosen
	independently and uniformly at random, so, by Theorem~\ref{thm:main2}, 
	$(S_{1,j}, \dots, S_{n,j})$ contains a transversal basis with probability at least $\nicefrac 12$. By the linearity of expectation, the expected number of disjoint transversal bases of $(B_1, \dots, B_n)$ is at least $\frac 12 \cdot 2m$. So there exist at least $m$ disjoint transversal bases.
	\end{proof}
	
	To prove Theorem \ref{thm:main2}, we use the following result of Rado~[\ref{cite:R42}] which characterizes the existence of a transversal basis.
	
	\begin{theorem} [Rado's Theorem]
		Let $(S_1, \dots, S_n)$ be sets of elements in a rank-$n$ matroid. Then there is a  transversal basis of $(S_1, \dots, S_n)$ if and only if $r(\cup_{i \in X} S_i) \geq |X|$ for all $X \subseteq \{1, \dots, n\}$.
	\end{theorem}
	
	In order to prove Theorem~\ref{thm:main2}, we will focus on the probability of failure of each of the conditions in Rado's Theorem.
	Let $B_1,\ldots, B_k$  be bases (not necessarily disjoint) of a rank-$n$ matroid and let $\alpha = 3\ceil{\log n}$.
	We let $Q(B_1,\ldots,B_k)$ denote the probability that, when $\alpha$-element subsets
	$S_1,\ldots,S_k$ are chosen independently and uniformly at random from $B_1,\ldots,B_k$, respectively,
	we have $r(S_1\cup\cdots\cup S_k)< k$.
	
	Note that we do not require the sets $B_1,\ldots,B_k$ to be disjoint. In fact, the case that
	$B_1=\cdots=B_k$ is interesting and plays an important role in the proof. In this case we have
	$r(S_1\cup\cdots\cup S_k) = |S_1\cup\cdots\cup S_k|$, and hence the failure probability $Q(B_1,\ldots,B_k)$
	depends only on $k$ and $n$; we let $Q_{k,n} = Q(B_1,\ldots,B_k)$. Thus 
	$Q_{k,n}$ denotes the probability that, when $\alpha$-element sets
	$S_1,\ldots,S_k$ are chosen independently and uniformly at random from the set $\{1,\ldots,n\}$
	we have $|S_1\cup\cdots\cup S_k|< k$.

	The following key lemma shows that the failure probability $Q(B_1,\ldots,B_k)$ is worst when $B_1=\cdots=B_k$; we postpone the proof of this 
	result until Section \ref{sec:details}.
	
	\begin{lemma} \label{lemma:worst-case-Rado}
		Let $n$ and $k$ be positive integers with $k \leq n$ and let $B_1, \dots, B_k$ be bases of a rank-$n$ matroid. Then $Q(B_1, \dots, B_k) \leq {Q_{k,n}}$.
	\end{lemma}
	
	Computing $Q_{k,n}$ is closely related to the Coupon Collector's Problem, as well as
	a bipartite matching problem considered by Erd\H os and Renyi~[\ref{cite:ER63}]. 
	
	\begin{lemma} \label{lemma:term-prob-bound}
		Let $n$ and $k$ be positive integers with $k \leq n$ and let $\alpha=3\ceil{\log n}$. Then
		$$
			{Q_{k,n}} \leq \binom{n}{k-1} \left( \frac{\binom{k-1}{\alpha}}{\binom{n}{\alpha}} \right)^k .
		$$
	\end{lemma}
	
	\begin{proof}
	There are ${n\choose \alpha}^k$ ways to choose $\alpha$-element sets
	$S_1,\ldots,S_k$ from $\{1,\ldots,n\}$. To bound the number of such $(S_1,\ldots,S_k)$
	with $|S_1\cup\cdots\cup S_k|< k$, we sum, over all $(k-1)$-element subsets $X$ of
	$\{1,\ldots,n\}$, the number of ways to choose $(S_1,\ldots,S_k)$ from $X$.
	\end{proof}
	
	Combining the above results gives us an upper bound on the failure probability in Theorem~\ref{thm:main2}.
	\begin{lemma} \label{lem:main2}
		Let $B_1, \dots, B_n$ be disjoint bases of a rank-$n$ matroid, where $n\ge 2$, and let $\alpha = 3 \ceil{\log n}$. If we choose $\alpha$-element sets $S_1 \subseteq B_1, \dots, S_n \subseteq B_n$ independently and uniformly at random, then the probability that $(S_1, \dots, S_n)$ {\em does not}
		contain a transversal basis is at most
		$$\sum_{k=1}^{n} \binom{n}{k} \binom{n}{k-1} \left(\frac{k-1}{n}\right)^{k\alpha}.$$
	\end{lemma}
	
	\begin{proof}
	By the union bound, the failure probability is at most the sum of the failure probabilities of each of the conditions in Rado's Theorem, so,
	by Lemmas~\ref{lemma:worst-case-Rado} and~\ref{lemma:term-prob-bound}, the probability that $(S_1, \dots, S_n)$  does not
       contain a transversal basis is at most
       $$\sum_{k=1}^{n} \binom{n}{k} \binom{n}{k-1}\left( \frac{\binom{k-1}{\alpha}}{\binom{n}{\alpha}} \right)^k .$$
       Moreover
       \begin{eqnarray*}
       \frac{\binom{k-1}{\alpha}}{\binom{n}{\alpha}}
       &=&\left(\frac{k-1}{n}\right)\left(\frac{k-2}{n-1}\right)\cdots\left(\frac{k-\alpha}{n-\alpha+1}\right)\\
       &\le&\left(\frac{k-1}{n}\right)^\alpha,
       \end{eqnarray*}
       since $k-1\le n$.
	\end{proof}
	Theorem~\ref{thm:main2} follows via a routine technical calculation which we complete in  Section~\ref{sec:details}.

	\section{Technical details} \label{sec:details}
	 
	 We start with the proof of Lemma~\ref{lemma:worst-case-Rado}.
	 \begin{proof}[Proof of Lemma~\ref{lemma:worst-case-Rado}.]
	 Let $B_1,\ldots, B_k$ be bases of a rank-$n$ matroid and let $\alpha = 3\ceil{\log n}$.
	Recall $Q_k(B_1, \dots, B_k)$ is the probability that $r(S_1\cup \cdots\cup S_k) < k$ in an experiment $\mathcal{E}$ where we choose $\alpha$-element subsets $S_1 \subset B_1, \dots, S_k \subset B_k$ independently and uniformly at random. We obtain a lower bound on $r(S_1\cup\cdots\cup S_k)$ 
	by constructing an independent set in a naive way. Given an outcome $(S_1, \dots, S_k)$ of $\mathcal{E}$, we construct bases $B'_1, \dots, B'_k$ and independent sets $I_1, \dots, I_k$ iteratively, such that:
	\begin{enumerate}[$\bullet$] 
	 	\item $B'_1 = B_1, I_1 = S_1$, and
	 	\item for each $i \in \{2, \dots, k\}$, the set $B'_i$ is an arbitrary basis with $I_{i-1} \subseteq B'_i \subseteq B_i \cup I_{i-1}$, and $I_i = I_{i-1} \cup (S_i \cap B'_i)$.
	\end{enumerate}
	 
	 Observe that the independent set $I_{i-1}$ can be extended to a basis $B'_i$ with $I_{i-1} \subseteq B'_i \subseteq B_i \cup I_{i-1}$
	 and that $I_i = I_{i-1} \cup (S_i \cap B'_i)\subseteq B'_i$, so $I_i$ is independent. Thus, given $(S_1, \dots, S_k)$,	 
	 the required bases $B'_1, \dots, B'_k$ and independent sets $I_1, \dots, I_k$ exist.
	 Note that $r(S_1\cup\cdots\cup S_k)\ge |I_k|$. It suffices to prove that $|I_k|<k$ with probability
	{\em equal to} $Q_{k,n}$. To see this we will describe an equivalent random process for generating 
	$B'_1,\ldots, B'_k$ and $I_1,\ldots, I_k$ based on a collection $(S'_1,\ldots,S'_k)$ of $\alpha$-element sets chosen
	independently and uniformly at random from $\{1,\ldots,n\}$ such that $|I_k|=|S'_1\cup\cdots\cup S'_k|$.
	 
	 We start with an observation regarding the construction of the sets $B'_1,\ldots,B'_k$ and $I_1,\ldots,I_k$.
	 Suppose, for some $i\ge 2$, we have already created $B'_1,\ldots,B'_{i-1}$ and $I_1,\ldots,I_{i-1}$.
	 We construct $B'_i$ by extending $I_{i-1}$ to a basis within $I_{i-1}\cup B_i$. Up to this point we have not used
	 the set $S_i$, so we may suppose that it is randomly generated at this time. Moreover we claim that, for the purpose of constructing $I_i$, we     may choose $S_i$ randomly from $B'_i$ instead of $B_i$. To see this, consider a bijection from $B_i$ to $B'_i$
	 that fixes the elements in $B_i \cap B'_i$, and let $S''_i$ denote the image of $S_i$ under this bijection. Since $B'_i-B_i \subseteq I_{i-1}$, we have
	 $I_{i-1} \cup (S_i \cap B'_i) = I_{i-1} \cup (S''_i \cap B'_i)$, so the set $I_i$, considered as a random variable,
	 has the same distribution when we choose $S_i$ from $B_i$ as it does when we choose $S_i$ from $B'_i$.
	 
	 In the following process, we will assume that sets $(S'_1,\ldots,S'_k)$ are only generated upon request. 
	 Initially we set $B'_1=B_1$ and choose an arbitrary bijection $\psi_1:\{1,\ldots,n\}\rightarrow B'_1$.
	 Now request $S'_1$. Note that $\psi_1(S_1')$ is chosen uniformly at random from the $\alpha$-element subsets of $B'_1$.
	 Set $S_1=\psi_1(S'_1)$ and $I_1=S_1$. For some $i\ge 2$, suppose that we have already created $B'_1,\ldots,B'_{i-1}$,
	 $\psi_1,\ldots,\psi_{i-1}$, $S_1,\ldots,S_{i-1}$ and $I_1,\ldots,I_{i-1}$. As before, we construct $B'_i$ by extending $I_{i-1}$ to a 
	 basis within $I_{i-1}\cup B_i$. Construct a bijection $\psi_i:\{1,\ldots,n\}\rightarrow B'_i$ such that
	 $\psi^{-1}_{i}(e)=\psi^{-1}_{i-1}(e)$ for all $e\in I_{i-1}$. Now request $S'_i$. Note that $\psi_i(S_i')$ is chosen uniformly at random from 
	 the $\alpha$-element subsets of $B'_i$. Set $S_i=\psi_i(S'_i)$ and $I_i=I_{i-1}\cup S_i$.
	 
	 A simple inductive argument shows that $|I_i| = |S'_1\cup\cdots\cup S'_i|$ for each $i\in\{1,\ldots,n\}$.
	 In particular, $|I_k| = |S'_1\cup\cdots\cup S'_k|$, as required.
	 \end{proof}
	  	
	It remains to prove Theorem~\ref{thm:main2}.
	
	\begin{proof}[Proof of Theorem~\ref{thm:main2}.]
	Let
	\[
		q_n =\sum_{k=1}^{n} \binom{n}{k} \binom{n}{k-1} \left(\frac{k-1}{n}\right)^{k\alpha}.
	\]
	
	By Lemma~\ref{lem:main2}, it suffices to prove that $q_n\le \nicefrac{1}{2}$.
	We have verified this numerically for all $n\in\{2,\ldots,59\}$ using Maple, so we may assume that 
	$n\ge 60$.
	
	Now we split the sum in two parts, change the index of summation in the second part, and apply the inequality
	$1+x\le e^x$, after which the terms in the two parts become identical.
	\allowdisplaybreaks{
	\begin{eqnarray*}
	q_n & = & \sum_{k=1}^{\lfloor\nicefrac{n}{2}\rfloor} \binom{n}{k} \binom{n}{k-1} \left(\frac{k-1}{n}\right)^{k\alpha} +\\
	 &&\quad \sum_{k=\lfloor\nicefrac{n}{2}\rfloor + 1}^{n} \binom{n}{k} \binom{n}{k-1} \left(\frac{k-1}{n}\right)^{k\alpha} \\
	 &= & \sum_{k=1}^{\lfloor\nicefrac{n}{2}\rfloor} \binom{n}{k} \binom{n}{k-1} \left(\frac{k-1}{n}\right)^{k\alpha} +\\
	 &&\quad \sum_{k=1}^{\lceil\nicefrac{n}{2}\rceil} \binom{n}{n-k+1} \binom{n}{n-k} \left(\frac{n-k}{n}\right)^{(n-k+1)\alpha} \\
	 &= & \sum_{k=1}^{\lfloor\nicefrac{n}{2}\rfloor} \binom{n}{k} \binom{n}{k-1} \left(1-\frac{n-k+1}{n}\right)^{k\alpha} +\\
	 &&\quad \sum_{k=1}^{\lceil\nicefrac{n}{2}\rceil } \binom{n}{k-1} \binom{n}{k} \left(1-\frac{k}{n}\right)^{(n-k+1)\alpha} \\
     &\le & \sum_{k=1}^{\lceil\nicefrac{n}{2}\rceil} \binom{n}{k}^2  e^{-\frac{n-k+1}{n} \cdot k\alpha} +\quad \sum_{k=1}^{\lceil\nicefrac{n}{2}\rceil }  \binom{n}{k}^2 e^{-\frac{k}{n}\cdot(n-k+1)\alpha} \\
	 &\le& 2 \sum_{k=1}^{\lceil\nicefrac{n}{2}\rceil} \left(\frac{en}{k}\right)^{2k}  e^{-\frac{n-k+1}{n} \cdot k\alpha},
	\end{eqnarray*}
	}
	as $\binom{n}{k} \leq \left(\frac{en}{k}\right)^k$. This bound is decreasing as a function of $\alpha$, so we can
	replace $\alpha$ with $3\log{n}$; after simplifying we get 
	\[
	q_n \leq 2 \sum_{k=1}^{\ceil{\nicefrac n2}} \left(\frac{e}{k}\right)^{2k} n^{-k+\frac{3k(k-1)}{n}}.
	\]
	
	Let $t_k =\left(\frac{e}{k}\right)^{2k} n^{-k+\frac{3k(k-1)}{n}}$. 
	\begin{claim} For each $k\in \{1,\ldots,\ceil{\frac n2}\}$ we have $t_k\le (\frac 12)^{k+2}$.
	\end{claim}
	
	\begin{proof}[Proof of claim.]
	We have numerically verified, for $k\in\{1,2,3\}$ and $n=60$, that $t_k\le \left(\frac 12\right)^{k+2}$
	(the $k=1$ case is where we require $n\ge 60$).
	Since $t_k$ is non-increasing as a function of $n$, 
	the claim holds for $k\in\{1,2,3\}$.
	
	Now consider the claim for $4 \leq k \leq \frac n3$. Note that when $k \geq 4$, the term $\left(\frac{e}{k}\right)^2$ can be bounded above by $\frac 1 2$. Furthermore, when $k \leq \frac n3$, we have $-k+\frac{3k(k-1)}{n}\le -1$. Hence,
	$$t_k \leq \left(\frac{e}{k}\right)^{2k}\frac 1n
	\le \left(\frac 1 2\right)^k \frac 1{60} < \left(\frac 12\right)^{k+2}.$$
	
	It remains to prove the claim for 
	$\frac n3 < k \leq \frac n2 +1$. Observe that $\frac{9e^2}{n^{\nicefrac 32-\nicefrac 3n}}$ is decreasing in $n$ when $n \geq 2$, so it is routine
	to verify that $\frac{9e^2}{n^{\nicefrac 32-\nicefrac 3n}}<\frac 12$ for all $n\ge 60$. Now,
	\begin{eqnarray*}
		t_k &=&  \left(\left(\frac{e}{k}\right)^2 n^{-1 + \frac{3k}{n}}\right)^k n^{-\frac {3k}{n}}\\
		&\leq&  \left(\left(\frac{e}{\nicefrac n3}\right)^2 n^{-1+\frac{3\left(\nicefrac n2 +1\right)}{n}}\right)^k n^{-1} \\
		&=& \left(\frac{9e^2}{n^{\nicefrac 32-\nicefrac 3n}}\right)^k n^{-1} \\
		&\leq& \left(\frac 12\right)^k\frac 1{60}\\
		&<&\left(\frac 12\right)^{k+2},
	\end{eqnarray*}
	as required.
	\end{proof}
	By the above claim,
	\[
	q_n \le 2\sum_{k= 1}^{\ceil{\nicefrac n2}} t_k 
	\le 2\sum_{k\ge 1} \left(\frac 12\right)^{k+2}
	= \frac 12,
	\]
	which completes the proof of Theorem \ref{thm:main2}. 
	\end{proof}

	\section*{Acknowledgment}
	We  thank Mike Molloy for helpful discussions regarding the probability calculations. We also thank the anonymous referees for their useful suggestions. 
	
	\section*{References}
	
	\newcounter{refs}
	
	\begin{list}{[\arabic{refs}]}%
		{\usecounter{refs}\setlength{\leftmargin}{10mm}\setlength{\itemsep}{0mm}}
	
	\item \label{cite:HR94}
	R. Huang, G.-C. Rota. On the Relations of Various Conjectures on Latin Squares and Straightening Coefficients. Discrete Math., 128 (1994), pp. 225-236.
	
	\item \label{cite:polymath12}
	Polymath 12. Rota's Basis Conjecture: Polymath 12. (2017),\\
	https://polymathprojects.org/2017/03/06/rotas-basis-conjecture-polymath-12-2/
	
	\item \label{cite:GW07}
	J. Geelen, K. Webb. On Rota's Basis Conjecture. SIAM Journal of Discrete Math, Vol. 21, No. 3 (2007), pp. 802-804.
	
	\item \label{cite:R42}
	R. Rado. A theorem on Independence Relations. The Quarterly Journal of Mathematics, Col. os-13, no. 1 (1942), pp. 83-89.
	
	\item \label{cite:ER63}
	P. Erd\H os, A. Renyi. On Random Matrices. Publications of the Mathematical Institute of the Hungarian Academy of Sciences. Vol. 8, Series A, fasc. 3 (1963), pp. 455-460.
	\end{list}
	
\end{document}